\documentclass[reqno]{amsart}%
\usepackage{amssymb}
\usepackage{amsmath}
\usepackage{amsfonts}
\usepackage{graphicx}
\usepackage{epstopdf}
\usepackage{epsfig}
\usepackage[pagewise]{lineno}
\usepackage{multicol}
\usepackage{subcaption}
\usepackage{caption}
\providecommand{\U}[1]{\protect\rule{.1in}{.1in}}
\theoremstyle{plain}

\numberwithin{equation}{section}

\newlength{\defbaselineskip}
\newcommand{\setlinespacing}[1]%
{\setlength{\baselineskip}{#1 \defbaselineskip}}

\usepackage{listings}
\newcommand{\norm}[1]{\left\Vert#1\right\Vert}

\newcommand{\unorm}[1]{\left\vert\left\vert\left\vert#1\right\vert\right\vert\right\vert}
\newtheorem{thm}{Theorem}[section]
\newtheorem{cor}[thm]{Corollary}
\newtheorem{lem}[thm]{Lemma}
\newtheorem{prop}[thm]{Proposition}

\newtheorem{property}[thm]{Property}
\newtheorem{rem}{Remark}[section]

\newcommand{\Complex}{\mathbb C}
\newcommand{\Real}{\mathbb R}

\makeatletter
\newcommand{\mathleft}{\@fleqntrue\@mathmargin0pt}
\newcommand{\mathcenter}{\@fleqnfalse}

\begin{document}
	\title
	{Generalized Numerical Radius Inequalities for Schatten p-Norms}
	
	\author{Jana Hamza}\
	\author{Hassan Issa}

	\address{Jana Hamza\newline  Department of Mathematics, Faculty of Arts and Sciences, Lebanese International University, Bekaa, Lebanon  } \email{21610074@students.liu.edu.lb}
	\address{Hassan Issa\newline  Department of Mathematics, Faculty of Science, Lebanese University, Beirut, Lebanon  } \email{hassan.issa@mathematik.uni-goettingen.de}	

	\begin{abstract}
	In this paper, we present various inequalities for generalized numerical radius of $2\times2$ block matrices for Schatten p-norm. Moreover, we give a refinement of the triangle inequality for the Schatten p-generalized  numerical radius.
		
	\end{abstract}
	
	\maketitle
\section{Introduction}

\hspace{\parindent} Consider the space $B(H)$ of bounded linear operators over a Hilbert space $H$.  For $A\in B(H)$, the numerical radius, the usual operator norm, and the Schatten p-norm, are denoted by $\omega{(A)}$, $\norm{A}$ and $\norm{A}_{p}$ respectively.\\
A norm $\unorm{\cdot}$ on $B(H)$ is said to be unitarily invariant if $\unorm{UAV}=\unorm{A}$, where $A\in B(H)$ and $U,V\in B(H)$ being unitary, and weakly unitarily invariant if $\unorm{UAU^{\ast}}=\unorm{A}$ where $A\in B(H)$ and $U\in B(H)$ being unitary. It is known that $\norm{A}_{p}$ is unitarily invariant.\\
We should note that if $B(H)$ is the space of $n\times n$ complex matrices, $M_{n}(\Complex)$, then for $A,B\in M_{n}(\Complex)$ we have
\begin{equation}\label{eq1}
	\unorm{ A\oplus A^{\ast}}
	=\unorm{ A\oplus A}
\end{equation}
and 

\begin{equation}\label{eq2}
	\unorm{ A\oplus B} =\unorm{
		\begin{bmatrix}
			0 & A \\
			B & 0
	\end{bmatrix} }
	\hspace{5pt}\mbox{(see \cite{[11]})}.
\end{equation}   
Moreover, we have

\begin{equation}\label{eqn5}
	\norm{ A\oplus B} _{p}=(\norm{ A}
	_{p}^{p}+\norm{ B} _{p}^{p})^{\frac{1}{p}},
\end{equation}
and 
\begin{equation}\label{eqn6}
	\norm{ A\oplus A} _{p} = 2^{\frac{1}{p}}\norm{ A}_{p}.
\end{equation}

The numerical radius $\omega{(\cdot)}$ is defined by $\omega{(A)}=\underset{\theta\in\Real}{\sup}\norm{Re(e^{i\theta}A)}$, where $A\in B(H)$. Due to its importance, the numerical radius has been generalized several times, and their last was given by Abu-Omar and Kittaneh \cite{[2]} in 2019, in which they generalized it on $B(H)$. This generalization is denoted  by $\omega_{N}(\cdot)$, and is defined by $\omega_{N}(A)=\underset{\theta\in\Real}{\sup}N(Re(e^{i\theta}A))$. It was proved by the authors in \cite{[2]} that $\omega_{N}(\cdot)$ generalizes the numerical radius $\omega(\cdot)$ if $N$ is the usual operator norm.

In this paper, we are interested in studying the space of $n\times n$ complex matrices, $M_{n}(\Complex)$. We should note that $\omega_{N}(\cdot)$ has the following two important properties, see \cite{[1]}.
\begin{property}\label{omegan}
	The following properties hold:
	\begin{itemize}
		\item[a)] The norm $\omega_{N}(\cdot)$ is self adjoint.
		\item[b)] If the norm $N(\cdot)$ is weakly unitarily invariant, then so is $\omega_{N}(\cdot)$.
	\end{itemize}
\end{property} 

In \cite{[6]}, Bhatia and Kittaneh where able to prove the following theorem that relates the Shatten p-norm of an $n\times n$ block matrix by that of its block entries.

\begin{thm}
	Let $T\in M_{n}(\mathbb{C} )$ such that $T=[T_{ij}]$, $1\leq i,j\leq n$ and $
	1\leq p\leq \infty $ then
	
	\begin{equation}\label{eqn3}
		n^{2-p}\norm{ T} _{p}^{p}\leq \overset{n}{\underset{i,j=1}{
				\sum }}\norm{ T_{ij}} _{p}^{p}\leq \norm{ T}
		_{p}^{p},
	\end{equation}
	for $2\leq p\leq \infty $; and
	
	\begin{equation}\label{eqn4}
		\norm{ T} _{p}^{p}\leq \overset{n}{\underset{i,j=1}{\sum }}
		\norm{ T_{ij}} _{p}^{p}\leq n^{2-p}\norm{ T}
		_{p}^{p},
	\end{equation}
	for $1\leq p\leq 2$.
\end{thm}

Motivated by the results of Bhatia and Kittaneh in \cite{[6]}, and those of Aldalabih and Kittaneh in \cite{[3]}, we aim in this paper to prove some generalized numerical radius inequalities for partitioned general $2\times 2$ block matrices considering the case when $N$ is taken to be the Schatten p-norm. We denote this norm by $\omega_{p}(\cdot)$ and call it the Schatten p-generalized numerical radius. We emphasize on finding such inequalities for the off-diagonal part of block matrices. We also provide an application of this norm in which we give a refinement of the triangle inequality for the Schatten p-generalized numerical radius. The following Lemma was proved by the authors in \cite{[1]}, and will be used in our work.
\begin{lem} \label{lemma1.4}
	Let $A$, $B\in M_{n}(\Complex)$, then
	\begin{equation*}
		\omega_{p}\left(
		\begin{bmatrix}
			0 & B \\
			B & 0
		\end{bmatrix}\right)=
		2^{\frac{1}{p}}\omega_{p}\left(B\right)
	\end{equation*}
for all p.
\end{lem}
\section{General $2 \times 2$ Block Matrices Inequalities}\label{sec8}
\hspace{\parindent}
In this section we give bounds for the generalized Schatten p-numerical radius of general $2\times2$ block matrices. We give emphasis for $2\times2$ block diagonal matrices. Most of the results in this section, generalize those presented in \cite{[3]}.

\begin{lem}\label{cor1}
	Let $A, B$ $\in$ $M$$_{n}(\Complex)$ then for all $p$, we have
	\begin{equation*}
		\omega_{p}\left(
		\begin{bmatrix}
			0 & A \\
			B & 0
		\end{bmatrix}
		\right) =2^{\frac{1}{p}-1}\underset{\theta \in\mathbb{R}}{\sup }\norm{
			e^{i\theta }A+e^{-i\theta }B^{\ast}} _{p}.
	\end{equation*}
	
\end{lem}

\begin{proof}
	By equations (\ref{eq1}), (\ref{eq2}) and (\ref{eqn6}), we have
	\begin{align*}
		\omega_{p}\left(
		\begin{bmatrix}
			0 & A \\
			B & 0
		\end{bmatrix}
		\right) &=\frac{1}{2}\underset{\theta \in \mathbb{R} }{\sup }\norm{
			\begin{bmatrix}
				0 & e^{i\theta }A+e^{-i\theta }B^{\ast } \\
				e^{i\theta }B+e^{-i\theta }A^{\ast } & 0
			\end{bmatrix}
		} _{p} \\
		&=\frac{1}{2}\underset{\theta \in \mathbb{R} }{\sup }\norm{
			\begin{bmatrix}
				e^{i\theta }A+e^{-i\theta }B^{\ast } & 0 \\
				0 & (e^{i\theta }A+e^{-i\theta }B^{\ast })^{\ast }
			\end{bmatrix}
		} _{p} \\
		&=\frac{1}{2}\underset{\theta \in \mathbb{R} }{\sup }\norm{
			\begin{bmatrix}
				e^{i\theta }A+e^{-i\theta }B^{\ast } & 0 \\
				0 & e^{i\theta }A+e^{-i\theta }B^{\ast }
			\end{bmatrix}
		} _{p} \\
		&=\frac{1}{2}\underset{\theta \in \mathbb{R} }{\sup }2^{\frac{1}{p}
		}\norm{ e^{i\theta }A+e^{-i\theta }B^{\ast}} _{p} \\
		&=2^{\frac{1}{p}-1}\underset{\theta \in \mathbb{R} }{\sup }\norm{
			e^{i\theta }A+e^{-i\theta }B^{\ast}} _{p}.
	\end{align*}
	
\end{proof}

\begin{prop}\label{prop}
	Let $A,B\in M_{n}(\mathbb{C})$, then the following inequality holds for all $p$
	\begin{equation*}
		\omega_{p}\left(
		\begin{bmatrix}
			A & 0 \\
			0 & B
		\end{bmatrix}\right)\leq \left(\omega_{p}^{p}(A)+\omega_{p}^{p}(B)\right)^{\frac{1}{p}}.
	\end{equation*} 
\end{prop}

\begin{proof}
	We have
	\begin{align*}
		\omega_{p}\left(
		\begin{bmatrix}
			A & 0 \\
			0 & B
		\end{bmatrix}\right) &= \underset{\theta \in \mathbb{R}}{\sup}\norm{ Re \left( e^{i\theta}
			\begin{bmatrix}
				A & 0 \\
				0 & B
			\end{bmatrix}\right)} _{p} \\
		&= \underset{\theta \in \mathbb{R}}{\sup}\norm{
			\begin{bmatrix}
				Re(e^{i\theta}A) & 0 \\
				0 & Re(e^{i\theta}B)
		\end{bmatrix}} _{p}\\
		&= \underset{\theta \in \mathbb{R}}{\sup}\left(\norm{ Re(e^{i\theta}A)}
		_{p}^{p}+\norm{ Re(e^{i\theta}B)} _{p}^{p}\right)^{\frac{1}{p}}\hspace{30pt} \mbox{(by equation (\ref{eqn5}))}\\
		&\leq\left(\omega_{p}^{p}(A)+\omega_{p}^{p}(B)\right)^{\frac{1}{p}}
	\end{align*}
	
	\noindent as required.
\end{proof}

\begin{thm}\label{thm}
	Let $A=[A_{ij}]$ be a $2 \times 2$ block matrix, then
	\begin{equation}
		\omega_{p}^{p}(A)\leq \frac{1}{2^{p-2}} \underset{i,j=1}{\overset{2}{\sum }}\omega_{p}^{p}(a_{ij})
	\end{equation}
	for $2\leq p\leq \infty $, and
	\begin{equation}
		\omega_{p}^{p}(A)\leq \underset{i,j=1}{\overset{2}{\sum }}\omega_{p}^{p}(a_{ij})
	\end{equation}
	for $1\leq p\leq 2$, where
	\begin{equation*}
		a_{ij}=\left\{
		\begin{array}{c}
			A_{ij}\hspace{1.5in}i=j \\
			2^{-\frac{1}{p}}
			\begin{bmatrix}
				0 & A_{ij} \\
				A_{ji} & 0
			\end{bmatrix}
			\hspace{0.5in}i\neq j.
		\end{array}
		\right.
	\end{equation*}
\end{thm}

\begin{proof}
	We have
	\begin{align*}
		\norm{ Re(e^{i\theta }A)} _{p} 
		&=\norm{
			\begin{bmatrix}
				Re(e^{i\theta }A_{11}) & \frac{1}{2}(e^{i\theta }A_{12}+e^{-i\theta
				}A_{21}^{\ast }) \\
				\frac{1}{2}(e^{i\theta }A_{21}+e^{-i\theta }A_{12}^{\ast }) & Re
				(e^{i\theta }A_{22})
			\end{bmatrix}
		} _{p},
	\end{align*}
	\noindent then by inequality
 (\ref{eqn3}) and Lemma \ref{cor1}, we get
	\begin{align*}
		\norm{ Re(e^{i\theta }A)} _{p}^{p} &\leq \frac{1}{2^{p-2}}
		\overset{2}{\underset{i,j=1}{\sum }}\norm{ (Re(e^{i\theta
			}A))_{ij}} _{p}^{p} \\
		&=\frac{1}{2^{p-2}}\left( \underset{i=j}{\sum }\norm{ Re(e^{i\theta
			}A_{ij})} _{p}^{p}+\underset{i\neq j}{\sum }\left(\frac{1}{2}
		\norm{ e^{i\theta }A_{ij}+e^{-i\theta }A_{ji}^{\ast }}
		_{p}\right)^{p}\right)  \\ 
		&\leq  \frac{1}{2^{p-2}}\left( \underset{i=j}{\sum }\omega_{p}^{p}(A_{ij})+\underset{i\neq j}{
			\sum }\left( 2^{-\frac{1}{p}}\omega_{p}\left(
		\begin{bmatrix}
			0 & A_{ij} \\
			A_{ji} & 0
		\end{bmatrix}
		\right) \right) ^{p}\right)  \\
		&=  \frac{1}{2^{p-2}}\underset{i=j}{\sum }\omega_{p}^{p}(a_{ij}).
	\end{align*}
	\noindent Then
	\begin{align*}
		\omega_{p}^{p}(A) &=\underset{\theta \in
			\mathbb{R}
		}{\sup }\norm{ Re(e^{i\theta }A)} _{p}^{p} \\
		&\leq \frac{1}{2^{p-2}}\underset{i=j}{\sum }\omega_{p}^{p}(a_{ij})
	\end{align*}
	\noindent for $2\leq p\leq \infty $. The second inequality is proved in a similar manner, using Lemma \ref{cor1}, and  inequality (\ref{eqn4}).
\end{proof}

\section{Off-Diagonal $2 \times 2$ Block Matrices Inequalities}\label{sec9}
\hspace{\parindent}
In this section, our interest was finding inequalities for $\omega_{p}$ of the off-diagonal $2 \times 2$ block matrices. The following lemma is useful in our work.

\begin{lem}\label{abcd}
	Let $A, B \in M_{n}(\mathbb{C})$, then
	\begin{itemize}
		\item[a)] $\omega_{p}\left(
		\begin{bmatrix}
			0 & A \\
			e^{i\theta}B & 0
		\end{bmatrix}\right)=
		\omega_{p}\left(\begin{bmatrix}
			0 & A \\
			B & 0
		\end{bmatrix}\right)$
		for all $\theta \in \mathbb{R}.$
		\item[b)] $\omega_{p}\left(
		\begin{bmatrix}
			0 & A \\
			B & 0
		\end{bmatrix}\right)=\omega_{p}\left(\begin{bmatrix}
			0 & B \\
			A & 0
		\end{bmatrix}\right).$
		
	\end{itemize}
	\noindent for all $p$.
\end{lem}

\begin{proof}
	Let $U=\begin{bmatrix}
		I & 0 \\
		0 & e^{i\frac{\theta}{2}}I
	\end{bmatrix}$, then $U$ is unitary. Then by Property \ref{omegan} we have
	\begin{align*}
		\omega_{p}\left(
		\begin{bmatrix}
			0 & A \\
			B & 0
		\end{bmatrix}\right) &= \omega_{p}\left(U
		\begin{bmatrix}
			0 & A \\
			B & 0
		\end{bmatrix}U^{\ast}\right) \\
		&= \omega_{p}\left(e^{-i\frac{\theta}{2}}
		\begin{bmatrix}
			0 & A \\
			e^{i\theta}B & 0
		\end{bmatrix}\right) \\
		&= \omega_{p}\left(
		\begin{bmatrix}
			0 & A \\
			e^{i\theta}B & 0
		\end{bmatrix}\right)
	\end{align*}
	\noindent which ends the proof of (a). Now to prove the equality (b), consider
	$U=\begin{bmatrix}
		0 & I \\
		I & 0
	\end{bmatrix}$, then $U$ is unitary.\newline
	Then by Property \ref{omegan} we have
	\begin{align*}
		\omega_{p}\left(
		\begin{bmatrix}
			0 & A \\
			B & 0
		\end{bmatrix}\right) &= \omega_{p}\left(U
		\begin{bmatrix}
			0 & A \\
			B & 0
		\end{bmatrix}U^{\ast}\right) \\
		&= \omega_{p}\left(
		\begin{bmatrix}
			0 & B \\
			A & 0
		\end{bmatrix}\right) 
	\end{align*}
	\noindent which ends the proof.
\end{proof}

The next theorem gives upper and lower bounds for $\omega_{p}\left(
\begin{bmatrix}
	0 & A \\
	B & 0
\end{bmatrix}\right)$ in terms of $\omega_{p}(A+B)$ and $\omega_{p}(A-B)$.

\begin{thm}\label{theorem}
	Let $A,B \in M_{n}(\Complex)$, then
	\begin{equation*}
		\frac{\max\left(\omega_{p}(A+B),\omega_{p}(A-B)\right)}{2^{1-\frac{1}{p}}}\leq
		\omega_{p}\left(
		\begin{bmatrix}
			0 & A \\
			B & 0
		\end{bmatrix}\right) \leq \frac{\omega_{p}(A+B)+\omega_{p}(A-B)}{2^{1-\frac{1}{p}}}
	\end{equation*} \noindent for all $p$.
\end{thm}

\begin{proof}
	By Lemma \ref{lemma1.4}, we have
	\begin{align*}
		2^{\frac{1}{p}}\omega_{p}(A+B) &= \omega_{p}\left(
		\begin{bmatrix}
			0 & A+B \\
			A+B & 0
		\end{bmatrix}\right) \\ 
		&= \omega_{p}\left(
		\begin{bmatrix}
			0 & B \\
			A & 0
		\end{bmatrix}+
		\begin{bmatrix}
			0 & A \\
			B & 0
		\end{bmatrix}\right)\\ 
		&\leq  \omega_{p}\left(
		\begin{bmatrix}
			0 & B \\
			A & 0
		\end{bmatrix}\right)+
		\omega_{p}\left(\begin{bmatrix}
			0 & A \\
			B & 0
		\end{bmatrix}\right)\hspace{30pt}\mbox{(by triangle inequality).} \\
		&= 2\omega_{p}\left(
		\begin{bmatrix}
			0 & A \\
			B & 0
		\end{bmatrix}\right)\hspace{30pt}\mbox{(by Lemma \ref{abcd}),}
	\end{align*} 
	\noindent then
	\begin{equation}\label{eqn7}
		\omega_{p}\left(
		\begin{bmatrix}
			0 & A \\
			B & 0
		\end{bmatrix}\right)\geq \frac{1}{2^{1-\frac{1}{p}}}\omega_{p}(A+B),
	\end{equation}
	replacing $B$ by $-B$ in inequality (\ref{eqn7}), we get
	\begin{equation*}
		\omega_{p}\left(
		\begin{bmatrix}
			0 & A \\
			-B & 0
		\end{bmatrix}\right)\geq \frac{1}{2^{1-\frac{1}{p}}}\omega_{p}(A-B),
	\end{equation*}
	taking $\theta=\pi$  in Lemma \ref{abcd}, we get
	\begin{align}
		\omega_{p}\left(
		\begin{bmatrix}
			0 & A \\
			B & 0
		\end{bmatrix}\right) &= \omega_{p}\left(
		\begin{bmatrix}
			0 & A \\
			-B & 0
		\end{bmatrix}\right) \nonumber \\
		&\geq \frac{1}{2^{1-\frac{1}{p}}}\omega_{p}(A-B)\label{eqn8},
	\end{align}
	\noindent therefore, by the estimations (\ref{eqn7}) and (\ref{eqn8}), we get
	\begin{equation*}
		\omega_{p}\left(
		\begin{bmatrix}
			0 & A \\
			B & 0
		\end{bmatrix}\right)\geq \frac{\max\left(\omega_{p}(A+B),\omega_{p}(A-B)\right)}{2^{1-\frac{1}{p}}}.
	\end{equation*}
	Now for the second inequality, consider $U=\frac{1}{\sqrt{2}}
	\begin{bmatrix}
		I & -I \\
		I & I
	\end{bmatrix}$,
	where $I$ is the $n\times n$ identity matrix, then $U$ is unitary, and thus by Property \ref{omegan}, we get
	\begin{align*}
		\omega_{p}\left(\begin{bmatrix}
			0 & A \\
			B & 0
		\end{bmatrix}\right) =& \omega_{p}\left(U
		\begin{bmatrix}
			0 & A \\
			B & 0
		\end{bmatrix}U^{\ast}\right) \\ 
		=& \frac{1}{2}\omega_{p}\left(\begin{bmatrix}
			-(A+B) & A-B \\
			-(A-B) & A+B
		\end{bmatrix}\right)\\
		\leq& \frac{1}{2}\left(\omega_{p}\left(\begin{bmatrix}
			-(A+B) & 0 \\
			0 & A+B
		\end{bmatrix}\right)
		+\omega_{p}\left(\begin{bmatrix}
			0 & A-B \\
			-(A-B) & 0
		\end{bmatrix}\right)\right)\hspace{10pt}\\
		{}&\mbox{(by triangle inequality)}\\
		\leq& \frac{1}{2}\left( 2^{\frac{1}{p}}\omega_{p}(A+B)+2^{\frac{1}{p}}\omega_{p}(A-B)\right) \\ {}& \mbox{(by Proposition \ref{prop}, Lemma \ref{abcd} and Lemma \ref{lemma1.4})}  	.
	\end{align*}
	\noindent as required.
\end{proof}

\begin{cor}
	Let $T\in M_{n}(\mathbb{C})$ such that $T=A+iB$, where $A=Re(T)$ and $B=Im(T)$, then
	\begin{equation*}
		\frac{\omega_{p}(T)}{2}\leq \frac{1}{2^{\frac{1}{p}}}\omega_{p}\left(
		\begin{bmatrix}
			0 & A \\
			B & 0
		\end{bmatrix}\right)\leq \omega_{p}(T).
	\end{equation*} for all $p$.
\end{cor}

\begin{proof}
	Replacing $B$ by $iB$ in theorem \ref{theorem}, we get
	\begin{equation*}
		\frac{\max\left(\omega_{p}(A+iB),\omega_{p}(A-iB)\right)}{2^{1-\frac{1}{p}}}\leq
		\omega_{p}\left(
		\begin{bmatrix}
			0 & A \\
			iB & 0
		\end{bmatrix}\right) \leq \frac{\omega_{p}(A+iB)+\omega_{p}(A-iB)}{2^{1-\frac{1}{p}}},
	\end{equation*}
	 then,
	\begin{equation*}
		\frac{\max\left(\omega_{p}(T),\omega_{p}(T^{\ast})\right)}{2^{1-\frac{1}{p}}}\leq
		\omega_{p}\left(
		\begin{bmatrix}
			0 & A \\
			iB & 0
		\end{bmatrix}\right) \leq \frac{\omega_{p}(T)+\omega_{p}(T^{\ast})}{2^{1-\frac{1}{p}}}.
	\end{equation*}
	However, $\omega_{p}(T)= \omega_{p}(T^{\ast})$, then
	\begin{equation*}
		\frac{\omega_{p}(T)}{2^{1-\frac{1}{p}}}\leq
		\omega_{p}\left(
		\begin{bmatrix}
			0 & A \\
			iB & 0
		\end{bmatrix}\right) \leq \frac{\omega_{p}(T)}{2^{-\frac{1}{p}}}.
	\end{equation*}
	Take $\theta = \frac{\pi}{2}$ in Lemma \ref{abcd}, then
	\begin{equation}\label{m}
		\frac{\omega_{p}(T)}{2^{1-\frac{1}{p}}}\leq
		\omega_{p}\left(
		\begin{bmatrix}
			0 & A \\
			B & 0
		\end{bmatrix}\right) \leq \frac{\omega_{p}(T)}{2^{-\frac{1}{p}}}.
	\end{equation}
	Multiply (\ref{m}) by $2^{-\frac{1}{p}}$, then
	\begin{equation*}
		\frac{\omega_{p}(T)}{2}\leq
		\frac{1}{2^{\frac{1}{p}}}\omega_{p}\left(
		\begin{bmatrix}
			0 & A \\
			B & 0
		\end{bmatrix}\right) \leq \omega_{p}(T).
	\end{equation*}
\end{proof}

\begin{rem}\label{rk1}
	Let $A \in M_{n}(\Complex)$, then we have
	\begin{equation*}
		e^{i(\theta -\frac{\pi}{2})} = -ie^{i\theta} \mbox{, and } e^{-i(\theta -\frac{\pi}{2})} = ie^{-i\theta}.
	\end{equation*}
	Therefore,
	\begin{align*}
		Re(e^{i(\theta - \frac{\pi}{2})}A) &= \frac{e^{i(\theta - \frac{\pi}{2})}A+e^{-i(\theta - \frac{\pi}{2})}A^{\ast}}{2} \\
		&= \frac{-ie^{i\theta}A+ie^{-i\theta}A^{\ast}}{2} \\
		&= \frac{e^{i\theta}A-e^{-i\theta}A^{\ast}}{2i} \\
		&= Im(e^{i\theta}A).
	\end{align*}
	\noindent  Therefore,
	\begin{align*}
		\omega_{p}(A) &= \underset{\alpha \in \mathbb{R}}{sup}\norm{ Re(e^{i\alpha}A)} _{p} \\
		&= \underset{\theta \in \mathbb{R}}{sup}\norm{ Re(e^{i(\theta - \frac{\pi}{2})}A)} _{p} \\
		&= \underset{\theta \in \mathbb{R}}{sup}\norm{ Im(e^{i\theta}A)} _{p}.
	\end{align*}
	See \cite{[8]}
\end{rem}

\begin{rem}\label{rk2}
	Let $X \in M_{n}(\Complex)$, and $2\leq p<\infty$ then
	\begin{align*}
		\omega_{p}\left(
		\begin{bmatrix}
			X & X \\
			-X & -X
		\end{bmatrix}\right) 
		=& \underset{\theta \in \mathbb{R}}{sup}\norm{
			\begin{bmatrix}
				Re(e^{i\theta}X) & Im(e^{i\theta}X) \\
				-Im(e^{i\theta}X) & -Re(e^{i\theta}X)
		\end{bmatrix}} _{p} \\
		\leq{}& \frac{1}{2^{\frac{2}{p}-1}}\underset{\theta \in \mathbb{R}}{sup}\left(2\norm{ Re(e^{i\theta}X)} _{p}^{p}+2\norm{ Im(e^{i\theta}X)} _{p}^{p}\right)^{\frac{1}{p}}  \\
		{}&\mbox{(by inequality (\ref{eqn3}))}\\
		=& \frac{1}{2^{\frac{2}{p}-1}}2^{\frac{2}{p}}\omega_{p}(X) \hspace{30pt}\mbox{(by Remark \ref{rk1})} \\
		=& 2\omega_{p}(X).
	\end{align*}
\end{rem}

\begin{thm}
	Let $A,B \in M_{n}(\Complex)$, then
	\begin{equation*}
		\omega_{p}\left(\begin{bmatrix}
			0 & A \\
			B & 0
		\end{bmatrix}\right)\leq
		2^{\frac{1}{p}}\min(\omega_{p}(A),\omega_{p}(B))+\min(\omega_{p}(A+B),\omega_{p}(A-B)).
	\end{equation*} for $2\leq p<\infty.$
\end{thm}

\begin{proof}
	Consider $U=\frac{1}{\sqrt{2}}\begin{bmatrix}
		I & I \\
		-I & I
	\end{bmatrix}$,
	where $I$ is the $n\times n$ identity matrix, then $U$ is unitary, we have
	\begin{align}
		\omega_{p}\left(\begin{bmatrix}
			0 & A \\
			B & 0
		\end{bmatrix}\right) =& \omega_{p}\left(U\begin{bmatrix}
			0 & A \\
			B & 0
		\end{bmatrix}U^{\ast}\right) \nonumber \\ 
		=& \frac{1}{2}\omega_{p}\left(\begin{bmatrix}
			A+B & A-B \\
			-(A-B) & -(A+B)
		\end{bmatrix}\right) \nonumber \\
		=& \frac{1}{2}\omega_{p}\left(\begin{bmatrix}
			A+B & A+B \\
			-(A+B) & -(A+B)
		\end{bmatrix}+
		\begin{bmatrix}
			0 & -2B \\
			2B & 0
		\end{bmatrix}\right) \nonumber \\ \displaybreak
		\leq{}& \frac{1}{2}\left(\omega_{p}\left(\begin{bmatrix}
			A+B & A+B \\
			-(A+B) & -(A+B)
		\end{bmatrix}\right)+
		\omega_{p}\left(\begin{bmatrix}
			0 & -2B \\
			2B & 0
		\end{bmatrix}\right)\right).\nonumber \\
		{}& \mbox{(by triangle inequality)}\nonumber\\
		\leq& \frac{1}{2}\left( 2\omega_{p}(A+B)+2^{\frac{1}{p}}\omega_{p}(2B)\right) \mbox{(by Remark \ref{rk2}, Lemma \ref{abcd} and Lemma \ref{lemma1.4})} \nonumber\\
		=& \omega_{p}(A+B)+2^{\frac{1}{p}}\omega_{p}(B). \label{eqn9}
	\end{align}

	\noindent Replacing $B$ by $-B$ in (\ref{eqn9}), we get
	\begin{align} 
		\omega_{p}\left(\begin{bmatrix}
			0 & A \\
			B & 0
		\end{bmatrix}\right) &= \omega_{p}\left(\begin{bmatrix}
			0 & A \\
			-B & 0
		\end{bmatrix}\right) \nonumber \hspace{30pt} \mbox{(taking $\theta = \pi$ in Lemma \ref{abcd})}\\
		&\leq \omega_{p}(A-B)+2^{\frac{1}{p}}\omega_{p}(B). \label{eqn10}
	\end{align}
	\noindent Then by (\ref{eqn9}) and (\ref{eqn10}), we get
	\begin{equation}\label{eqn11}
		\omega_{p}\left(\begin{bmatrix}
			0 & A \\
			B & 0
		\end{bmatrix}\right)\leq 2^{\frac{1}{p}}\omega_{p}(B)+\min(\omega_{p}(A+B),\omega_{p}(A-B)).
	\end{equation}
	Interchanging $A$ and $B$ in (\ref{eqn11}), we get
	\begin{align}
		\omega_{p}\left(\begin{bmatrix}
			0 & A \\
			B & 0
		\end{bmatrix}\right) &= \omega_{p}\left(\begin{bmatrix}
			0 & B \\
			A & 0
		\end{bmatrix}\right)\nonumber \hspace{30pt} \mbox{(by Lemma \ref{abcd})} \\
		&\leq 2^{\frac{1}{p}}\omega_{p}(A)+\min(\omega_{p}(A+B),\omega_{p}(A-B)). \label{eqn12}
	\end{align}
	\noindent Therefore, by (\ref{eqn11}) and (\ref{eqn12}), we have
	\begin{equation*}
		\omega_{p}\left(\begin{bmatrix}
			0 & A \\
			B & 0
		\end{bmatrix}\right)\leq 2^{\frac{1}{p}}\min(\omega_{p}(A),\omega_{p}(B))+\min(\omega_{p}(A+B),\omega_{p}(A-B))
	\end{equation*}
	as required.\\
\end{proof}

\section{An Application}
In this section we present an application, which is a refinement of the triangle inequality for the generalized Schatten p-numerical radius. The following remark is presented by "Yamazaki" in \cite{[14]}

\begin{rem}\label{rk3}
	Let $T \in M_{n}(\Complex)$, we have
	\begin{equation*}
		\omega_{p}(T)= \underset{\alpha^{2}+\beta^{2}=1}{sup}\norm{\alpha Re(T)+\beta Im(T)}_{p}.
	\end{equation*}
	for all $p$. Then
	\begin{equation*}
		\norm{T+T^{\ast}}_{p}\leq 2\omega_{p}(T).
	\end{equation*}
\end{rem}

\begin{proof}
	We have
	\begin{align*}
		\omega_{p}(T) =& \underset{\theta \in \mathbb{R}}{sup}\norm{Re(e^{i\theta}T)}_{p} \\
		=& \frac{1}{2}\underset{\theta \in\mathbb{R}}{sup}\norm{e^{i\theta}T+e^{-i\theta}T^{\ast}}_{p} \\
		=& \underset{\theta \in\mathbb{R}}{sup}\norm{\cos(\theta) T+i\sin( \theta) T+\cos(\theta) T^{\ast}-i\sin(\theta) T^{\ast}}_{p} \\
		=& \underset{\theta \in\mathbb{R}}{sup}\norm{\cos(\theta) Re(T)-\sin(\theta) Im(T)}_{p} \\
		=& \underset{\alpha^{2}+\beta^{2}=1}{sup}\norm{\alpha Re(T)+\beta Im(T)}_{p}.
	\end{align*}
	\noindent  Take $\theta=2\pi$, then
	\begin{align*}
		\omega_{p}(T) &\geq \norm{\cos(2\pi) Re(T)-\sin(2\pi) Im(T)}_{p} \\
		&= \frac{1}{2}\norm{T+T^{\ast}}_{p}
	\end{align*}
	\noindent so, $\norm{T+T^{\ast}}_{p}\leq 2\omega_{p}(T)$.
\end{proof}

The next theorem is a refinement of the triangle inequality for the generalized Schatten p-numerical radius.

\begin{thm}\label{application}
	Let $A,B \in M_{n}(\Complex)$, then
	\begin{equation*}
		\norm{ A+B }_{p}\leq 2^{1-\frac{1}{p}}\omega_{p}\left(\begin{bmatrix}
			0 & A \\
			B^{\ast} & 0
		\end{bmatrix}\right)\leq
		\norm{A}_{p}+\norm{B}_{p}.
	\end{equation*} for all $p$.
\end{thm}

\begin{proof}
	Let $T=\begin{bmatrix}
		0 & A \\
		B^{\ast} & 0
	\end{bmatrix}$.
	We have
	\begin{align}
		\norm{\begin{bmatrix}
				0 & A+B \\
				A^{\ast}+B^{\ast} & 0
		\end{bmatrix}}_{p}^{p} =& \norm{(A+B)\oplus(A^{\ast}+B^{\ast})}_{p}^{p}\nonumber \hspace{30pt} \mbox{(by equation (\ref{eq2})}\\
		=& \norm{(A+B)\oplus(A+B)}_{p}^{p} \nonumber \hspace{30pt} \mbox{(by equation (\ref{eq1}))} \\
		=& 2 \norm{ A+B}_{p}^{p}. \hspace{30pt} \mbox{(by equation (\ref{eqn6}))} \label{eq}
	\end{align}
	\noindent  Then,
	\begin{align*}
		2 \norm{ A+B}_{p}^{p} =& \norm{\begin{bmatrix}
				0 & A+B \\
				A^{\ast}+B^{\ast} & 0
		\end{bmatrix}}_{p}^{p} \\
		=& \norm{ T+T^{\ast}}_{p}^{p} \\
		\leq& 2^{p}\omega_{p}^{p}(T) \hspace{30pt} \mbox{(by Remark \ref{rk3})} \\
		=& 2^{p}\omega_{p}^{p}\left(\begin{bmatrix}
			0 & A \\
			B^{\ast} & 0
		\end{bmatrix}\right).
	\end{align*}
	\noindent Thus,
	\begin{equation*}
		\norm{ A+B}_{p}\leq 2^{1-\frac{1}{p}}\omega_{p}\left(\begin{bmatrix}
			0 & A \\
			B^{\ast} & 0
		\end{bmatrix}\right).
	\end{equation*}
	For the second inequality, we have
	\begin{align*}
		\omega_{p}(T) =& \underset{\theta \in \mathbb{R}}{sup}\norm{ Re(e^{i\theta}T)}_{p} \\ 
		=& \frac{1}{2}\underset{\theta \in \mathbb{R}}{sup}\norm{ \begin{bmatrix}
				0 & e^{i\theta}A+e^{-i\theta}B \\
				e^{-i\theta}A^{\ast}+e^{i\theta}B^{\ast} & 0
		\end{bmatrix}}_{p} \\ 
		=& \frac{1}{2^{1-\frac{1}{p}}}\underset{\theta \in \mathbb{R}}{sup}\norm{ e^{i\theta}A+e^{-i\theta}B}_{p} \hspace{30pt} \mbox{(by same argument as (\ref{eq}))}\\
		\leq& \frac{1}{2^{1-\frac{1}{p}}}\left(\norm{ A}_{p}+\norm{ B}_{p}\right).\hspace{30pt} \mbox{(by triangle inequality)}
	\end{align*}
	\noindent Therefore,
	\begin{equation*}
		\norm{ A+B}_{p}\leq 2^{1-\frac{1}{p}}\omega_{p}\left(\begin{bmatrix}
			0 & A \\
			B^{\ast} & 0
		\end{bmatrix}\right)\leq
		\norm{ A}_{p}+\norm{ B}_{p}.
	\end{equation*}
\end{proof}

\end{document}